\documentclass[envcountsect,referee]{svjour3}

\smartqed  
\usepackage{graphicx}
\usepackage{latexsym,bm,amsmath,amssymb,mathrsfs}
\usepackage{rotating}
\usepackage{multirow,booktabs,cases}
\usepackage[misc]{ifsym}
\usepackage[style=1]{mdframed}
\usepackage{algorithm,algorithmic}
\usepackage[colorlinks=true]{hyperref}
\hypersetup{urlcolor=blue,citecolor=blue,linkcolor=blue}
\usepackage{epstopdf}

\def\[{\begin{equation}}
\def\]{\end{equation}}

\newtheorem{assum}{Assumption}[section]
\newtheorem{exam}{Example}[section]
\newtheorem{pro}{Property}[section]

\numberwithin{equation}{section}
\allowdisplaybreaks

\begin{document}
\graphicspath{{./FIG/}}

\title{On the properties of tensor complementarity problems}


\author{Wen Yu \and Chen Ling \and Hongjin He}


\institute{W. Yu \and C. Ling\and H. He \at
Department of Mathematics, School of Science, Hangzhou Dianzi University, Hangzhou, 310018, China.\\
\email{yuwen\_hdu@163.com}
\and C. Ling \at
\email{cling\_zufe@sina.com}
\and H. He (\Letter) \at
\email{hehjmath@hdu.edu.cn}
 }

\date{Received: date / Accepted: date}

\maketitle

\begin{abstract}
Properties of solutions of the {\it tensor complementarity problem} (TCP) for structured tensors have been investigated in recent literature. In this paper, we make further contributions on this problem. Specifically, we first derive solution existence theorems for TCPs on general cones from the results studied in the {\it nonlinear complementarity problem} literature. An interesting byproduct is that conditions (e.g., strict copositivity) of solution existence results for TCPs on the nonnegative cone can be reduced to copositivity, which, to the best of our knowledge, is the weakest requirement in the current TCP literature. Moreover, we study the topological properties of the solution set and stability of the TCP at a given solution, which are not discussed before and further enrich the theory of TCPs.



\end{abstract}

\section{Introduction}\label{Introduction}
The {\it complementarity problem} is a long historical topic that has become a well-established and fruitful discipline within mathematical programming. We here refer the reader to monographs \cite{CPS92,FP03,Isc92} and surveys \cite{FP97,HP90b} for the well developed basic theory, numerical algorithm, and applications of complementarity problems.

In recent decades, with the rapid development of the discipline of tensors, the so-called {\it tensor complementarity problem} (TCP) over the nonnegative cone was introduced recently for the research on structured tensors \cite{SQ14}. It is well known that the concept of tensors is a natural generalization of matrices. Therefore, the classical {\it linear complementarity problem} (LCP) is a special case of TCPs. However, most of the well established properties of LCPs cannot be extended to TCPs directly due to the complicated structure of tensors. According to the definition of TCP, we can easily see that such a model falls into a special case of the {\it nonlinear complementarity problem} (NCP), and thus many results of NCPs are certainly applicable to TCPs, which, however, do not often embody the structure of tensors. Comparatively, we are more interested in some specialized properties of TCPs by fully considering the structure of tensors. For instance, some recent papers are dedicated to showing the existence of solutions of TCPs with some structured tensors (e.g., nonnegative tensors \cite{HSW15,SQ14}, symmetric positive definite tensors and copositive tensors \cite{CQW16}, (strictly-) semi-positive tensors \cite{SQ16,WHB16}, Z-tensors \cite{GLQX15} and M-tensors \cite{DQW13,ZQZ14}, ER-tensors \cite{BHW16}, and P-tensors \cite{BHW16,DLQ15,SQ14}). In \cite{BHW16}, the authors considered the property of global uniqueness and solvability for TCPs with particular tensors. Along with the booming of sparse optimization, Luo et al. \cite{LQX15} studied the sparsest solutions to TCPs. Most recently, Huang and Qi \cite{HQ16} reformulated an $n$-person noncooperative game as a TCP showing an interesting application of TCPs in management science.

Note that all results mentioned above focus on the special case, TCPs over the nonnegative cone, and require some relatively stronger conditions. Therefore, there are three natural questions: i) can we consider the TCP on a general cone? ii) can we establish the solution existence theorem under weaker conditions on tensors than the previous results? iii) what more properties can we obtain for the problem under consideration?

In this paper, we consider the TCP over a general cone extending the model introduced in \cite{SQ14}. As an special case of NCPs, we derive some specific solution existence results for the TCPs under consideration from the results presented by Gowda and Pang \cite{GP92}. An interesting consequence is that conditions (e.g., strict copositivity) of the solution existence for TCPs over the nonnegative cone can be weakened to copositivity, which is weaker than the requirements investigated in the TCP literature. The other two important results are the topological properties of the solution set and stability of solutions of the general TCP. To the best of our knowledge, such properties are not studied in the current TCP literature. All results presented in this paper further enrich the theory of TCPs.

The structure of this paper is organized as follows. In Section \ref{probdes}, we first describe the general model of the TCP under consideration and introduce some notations that will be used throughout. In Section \ref{Prelim}, we summarize some definitions and properties, which are preparations of the subsequent analysis. In Section \ref{Existence}, based on the results of Gowda and Pang \cite{GP92}, we present the existence of solutions of TCPs under mild conditions. In Section \ref{sectop}, we show the topological properties and stability of solutions of general TCPs. Finally, we complete this paper with drawing some concluding remarks in Section \ref{SecCon}.

\section{The Model and Notation}\label{probdes}
The concept of tensors is a natural generalization of matrices. Notationally,
let $\mathcal{A}:=(a_{i_1i_2\ldots i_m})_{1\leq i_1, i_2,\ldots,i_m\leq n}$ denote an $m$-th order $n$-dimensional square tensor, where $a_{i_1i_2\ldots i_m}\in \mathbb{R}$. Denote by ${\mathcal T}_{m,n}$ the space of all $m$-th order $n$-dimensional square tensors. Clearly, we can see that ${\mathcal T}_{m,n}$ is a linear space of dimension $n^m$. Here, we shall mention that all the tensors discussed in this paper are real. Besides, we denote by $\mathcal{I}:=(\sigma_{i_1\cdots i_m})$ the unit tensor in ${\mathcal T}_{m,n}$, where $\sigma_{i_1\cdots i_m}$ is the Kronecker symbol
$$
\sigma_{i_1\cdots i_m}:=\left\{
\begin{array}{ll}
1,&\;\;{\rm if~}i_1=\cdots =i_m,\\
0,&\;\;{\rm otherwise}.
\end{array}
\right.
$$

For given $\mathcal{A}\in {\mathcal T}_{m,n}$ and ${\bf q}\in \mathbb{R}^n$, the {\it tensor complementarity problem} (TCP) refers to the task of finding a vector ${\bf x}\in \mathbb{R}^n$ such that
\begin{equation}\label{TCP}
{\bf x}\in K,~{\bf w}:=\mathcal{A}{\bf x}^{m-1}+{\bf q}\in K^*,~~~{\rm and}~~~{\bf x}\perp{\bf w}:=\langle {\bf x},{\bf w}\rangle=0,
\end{equation}
where $K\subset\mathbb{R}^n$ is a given closed and convex pointed cone, $K^*$ is the dual cone of $K$ defined by
$$
K^*:=\left\{{\bf y}\in \mathbb{R}^n~|~\langle {\bf y},{\bf x}\rangle \geq 0,~\forall~{\bf x}\in K\right\},
$$
and $\langle\cdot,\cdot\rangle$ denotes the standard inner product in real Euclidean space, the function $F({\bf x}):=\mathcal{A}{\bf x}^{m-1}:~{\mathbb R}^n\to{\mathbb R}^n$ whose $i$-th component is given by
\begin{equation}\label{tenx}
F_i({\bf x}):=(\mathcal{A}{\bf x}^{m-1})_i:=\sum_{i_2,\ldots,i_m=1}^na_{ii_2\ldots i_m}x_{i_2}\cdots x_{i_m}.
\end{equation}
Throughout, we define $\mathcal{A}{\bf x}^m$ as the value at ${\bf x}$ of a homogeneous polynomial of the form
$$\mathcal{A}{\bf x}^{m}:={\bf x}^\top \mathcal{A}{\bf x}^{m-1}:=\sum_{i_1,i_2,\ldots,i_m=1}^na_{i_1i_2\cdots i_m}x_{i_1}x_{i_2}\cdots x_{i_m}.$$
Obviously, each component of $\mathcal{A}{\bf x}^{m-1}$ is a homogeneous polynomial of degree $m-1$. Note that TCP \eqref{TCP} immediately reduces to the model studied in \cite{BHW16,SQ14,SQ16} when taking $K:={\mathbb R}^n_{+}$. In what follows, we denote TCP \eqref{TCP} by ${\rm TCP}(K, {\bf q},\mathcal{A})$ for notational convenience. Moreover, we denote by SOL($K, {\bf q},\mathcal{A}$) the solution set of TCP($K, {\bf q},\mathcal{A}$), i.e.,
$$ {\rm SOL}(K,{\bf q},\mathcal{A}):=\left\{{\bf x}\in \mathbb{R}^{n}~|~K\ni{\bf x}\perp(\mathcal{A}{\bf x}^{m-1}+{\bf q})\in K^*\right\}.$$
It is noteworthy that the solution set ${\rm SOL}(K, {\bf q}, \mathcal{A})$ is possibly empty for general tensors. Therefore, it is necessary to investigate that what tensors could make the solution set ${\rm SOL}(K, {\bf q}, \mathcal{A})$ nonempty and what more properties we can get for such tensors.

Throughout this paper, let $I_n := \{1, 2, \cdots, n\}$ be an index set. Denote $\mathbb{R}^n$ the real Euclidean space of column vectors with length $n$, i.e., $$\mathbb{R}^n := \left\{{\bf x}=(x_1, x_2,\ldots, x_n)^\top~|~ x_i\in \mathbb{R}, i\in I_n\right\},$$ where $\mathbb{R}$ is the set of real numbers and the symbol $^\top$ represents the transpose. Correspondingly, $\mathbb{R}_+^n:=\{{\bf x}\in \mathbb{R}^n~|~{\bf x}\geq {\bf 0}\}$ and $\mathbb{R}_{++}^n:=\{{\bf x}\in \mathbb{R}^n~|~{\bf x}> {\bf 0}\}$, where ${\bf x}\geq {\bf 0}~(>{\bf 0})$ means $x_i\geq 0~(x_i>0)$ for all $i\in I_n$. For ${\bf u}\in \mathbb{R}^n$ and ${\bf v}\in \mathbb{R}^m$, we denote by $({\bf u},{\bf v})$ the column vector $({\bf u}^\top,{\bf v}^\top)^\top$ in $\mathbb{R}^{n+m}$ for simplicity. For a vector ${\bf x}\in \mathbb{R}^n$ and a real number $p\geq 0$, denote ${\bf x}^{[p]}:=(x_1^p,x_2^p,\ldots, x_n^p)^\top$. Let ${\bf e}_n:=(1,1,\ldots,1)^\top$ with length $n$. Let $\mathbb{B}_n({\bf x},r)$ represent the closed ball centered at ${\bf x}$ with radius
$r$ in ${\mathbb R}^n$, and in particular, denote by $\mathbb{B}_n$ the unit sphere centered at ${\bf 0}\in {\mathbb R}^n$. For a given subset ${\mathcal N}$ in $\mathbb{R}^n$, denote by ${\rm cl}({\mathcal N})$ the topological closure of ${\mathcal N}$.

For given $\mathcal{A}\in {\mathcal T}_{m,n}$ and a nonempty subset $\alpha$ of $I_n$, we denote the principal sub-tensor of $\mathcal{A}$ by $\mathcal{A}_\alpha$, which is obtained by homogeneous polynomial $\mathcal{A}{\bf x}^m$ for all ${\bf x}:=(x_1,x_2,\ldots,x_n)^\top$ with $x_i=0$ for $\bar\alpha=I_n\backslash \alpha$. So, $\mathcal{A}_\alpha\in {\mathcal T}_{m,|\alpha|}$, where the symbol $|\alpha|$ denotes the cardinality of $\alpha$. Correspondingly, we denote the sub-tensor of $\mathcal{A}$ by $\mathcal{A}_{\bar \alpha\alpha}$, which consists of the elements $a_{i_1i_2\ldots i_m}$ in $\mathcal{A}$ with $i_1\not\in \alpha$ and $i_2,\ldots,i_m\in \alpha$. Notice that $\mathcal{A}_{\bar \alpha\alpha}$ is not a square tensor. If the entries $a_{i_1i_2\ldots i_m}$ of a given tensor ${\mathcal A}$ are invariant under any permutation of their indices, then $\mathcal{A}$ is called a {\it symmetric} tensor. If for every $i\in I_n$, $\mathcal{A}_i:=(a_{ii_2\ldots i_m})_{1\leq i_2,\ldots,i_m\leq n}$, an $(m-1)$-th order $n$-dimensional square tensor, is symmetric, then $\mathcal{A}$ is called a {\it sub-symmetric} tensor with respect to the indices $\{i_2,\ldots,i_m\}$.  Apparently, a symmetric tensor $\mathcal{A}$ must be sub-symmetric, but the reverse is not true. For given $\mathcal{A}:=(a_{i_1i_2\ldots i_m})\in \mathcal{T}_{m,n}$ and ${\bf x}\in \mathbb{R}^n$, $\mathcal{A} {\bf x}^{m-2}$ denotes the $n\times n$ matrix with its $(i,j)$-th element given by
$$\left(\mathcal{A} {\bf x}^{m-2}\right)_{ij}:=\sum_{i_3,\ldots,i_m=1}^na_{iji_3\ldots i_m} x_{i_3}\cdots  x_{i_m}.$$

\section{Definitions and Lemmas}\label{Prelim}

In this section, we introduce some basic definitions and lemmas, which pave the way
of our further analysis.

Let $K$ be a closed and convex pointed cone in $\mathbb{R}^n$. Recall that a nonempty set $\Omega\subset \mathbb{R}^n$ generates $K$, thereby writing $K:={\rm cone}(\Omega)$ if $K:=\{t {\bf s}\;:\;{\bf s}\in \Omega,\;t\in \mathbb{R}_+\}$. If in addition $\Omega$ does not contain the zero vector and for each ${\bf x}\in K\backslash \{{\bf 0}\}$, there exist unique ${\bf s}\in \Omega$ and $t\in \mathbb{R}_+$ such that ${\bf x}=t{\bf s}$, then we say that $\Omega$ is a {\it basis} of $K$. Whenever $\Omega$ is a finite set, ${\rm cone}({\rm conv}(\Omega))$ is called a {\it polyhedral} cone, where ${\rm conv}(\Omega)$ stands for the {\it convex hull} of $\Omega$. It is clear that $\mathbb{R}_+^n$ is a closed convex cone in $\mathbb{R}^n$, whose a compact basis is $\Delta:=\{{\bf x}\in \mathbb{R}_+^n~|~{\bf e}_n^\top {\bf x}=1\}$.

\begin{definition}\label{Kpositive} Let $\mathcal{A}=(a_{i_1i_2\ldots i_m})\in {\mathcal T}_{m,n}$, and let $K$ be a given closed convex cone in $\mathbb{R}^n$. Then, $\mathcal{A}$ is said to be
\begin{itemize}
\itemindent 8pt
\item[\rm(i)] $K$-positive semi-definite, if $\mathcal{A}{\bf x}^m\geq 0$ for any vector ${\bf x}\in K$;

\item[\rm(ii)] $K$-positive definite, if $\mathcal{A}{\bf x}^m> 0$ for any vector ${\bf x}\in K\backslash\{{\bf 0}\}$.
\end{itemize}
In particular, an $\mathbb{R}_+^n$-positive semi-definite ($\mathbb{R}_+^n$-positive definite) tensor is called {\it copositive} ({\it strictly copositive}) tensor.
\end{definition}

The following property characterizes $K$-positive semi-definite (definite) tensors and extends ones proposed in \cite{Qi13}. For the sake of completeness, here we still present its proof.

\begin{pro}
 Let $K$ be a closed convex cone associated with a compact basis $\Omega$.
Let $\mathcal{A}\in {\mathcal T}_{m,n}$. Then, $\mathcal{A}$ is $K$-positive semi-definite (definite) if
and only if
$$\min\left\{\mathcal{A}{\bf x}^m ~|~ {\bf x}\in \Omega\right\}\geq 0 ~(>0).
$$
\end{pro}

\begin{proof}
For every ${\bf x}\in K\backslash\{{\bf 0}\}$, there exist unique ${\bf s}\in \Omega$ and $t\in \mathbb{R}_{++}$ such that ${\bf x}=t{\bf s}$. Consequently, the desired result follows from Definition \ref{Kpositive}.
\qed\end{proof}

\begin{definition}\label{Kregular}
Let $\mathcal{A}\in {\mathcal T}_{m,n}$, and let $K$ be a given closed convex pointed cone in $\mathbb{R}^n$. We say that $\mathcal{A}$ is $K$-regular if it satisfies
$$
\mathcal{A}{\bf x}^m\neq 0,~~ \forall~ {\bf x}\in K\backslash \{{\bf 0}\}.
$$
\end{definition}

\begin{definition}\label{NonsingularDef}
 Let $\mathcal{A}\in {\mathcal T}_{m,n}$, and let $K$ be a given closed convex cone in $\mathbb{R}^n$. We say that $\mathcal{A}$ is $K$-singular if it satisfies
 $$\{{\bf x}\in K\backslash\{{\bf 0}\}~|~\mathcal{A}{\bf x}^{m-1}={\bf 0}\}\neq \emptyset.$$
 Otherwise, we say that $\mathcal{A}$ is $K$-nonsingular. In particular, we say that $\mathcal{A}$ is singular if it satisfies
 $$\{{\bf x}\in \mathbb{C}^n\backslash\{{\bf 0}\}~|~\mathcal{A}{\bf x}^{m-1}={\bf 0}\}\neq \emptyset.$$
 Otherwise, $\mathcal{A}$ is said to be nonsingular.

\end{definition}

From Definition \ref{NonsingularDef}, it is easy to see that, $\bar {\mathcal{A}}$ is $\bar K$-nonsingular if and only if, for every sequence $\{x^{(l)}\}\subset \bar K$ with $\|x^{(l)}\|\rightarrow \infty$ as $l\rightarrow\infty$, there exists a subsequence $\{x^{(l_i)}\}$ of $\{x^{(l)}\}$ such that $\|\bar {\mathcal{A}}(x^{(l_i)})^{m-1}\|\rightarrow \infty$ as $l_i\rightarrow \infty$.
It is clear that, if $\mathcal{A}$ is nonsingular, then for any given closed convex cone $K$ in $\mathbb{R}^n$, $\mathcal{A}$ is $K$-nonsingular.  From Definitions \ref{Kpositive} and \ref{NonsingularDef}, it is easy to see that, a $K$-positive definite tensor $\mathcal{A}$ must be $K$-nonsingular. As we know, if a symmetric matrix $A\in {\mathcal T}_{2,n}$ is positive semi-definite and nonsingular, then it must be positive definite. However, if $A$ is asymmetric, the above conclusion is not true.  The following example shows that, when $m\geq 3$, even if $\mathcal{A}\in {\mathcal T}_{m,n}$ is symmetric, $K$-positive semi-definiteness and $K$-nonsingularity of $\mathcal{A}$ do not imply the $K$-positive definiteness of $\mathcal{A}$.
\begin{exam}
Let $\mathcal{A}=(a_{i_1i_2i_3})\in {\mathcal T}_{3,2}$, where $a_{122}=a_{212}=a_{221}=a_{211}=a_{121}=a_{112}=1$ and all other $a_{i_1i_2i_3}=0$. It is clear that $\mathcal{A}$ is symmetric. Moreover, for any ${\bf x}\in \mathbb{R}^2$, we have
$$\mathcal{A}{\bf x}^2=\left[
\begin{array}{c}
x^2_2+2x_1x_2\\
x^2_1+2x_1x_2
\end{array}
\right]\quad\text{and}\quad\mathcal{A}{\bf x}^3=3x_1x_2^2+3x_1^2x_2.$$
Consequently, it is easy to verify that $\mathcal{A}$ is copositive and nonsingular, but not strictly copositive.
\end{exam}

Denote by $C(\mathbb{R}^n)$ the set of nonzero closed convex cones in $\mathbb{R}^n$, which is associated with the natural metric defined by
$$
\delta(K_1,K_2) := \sup_{\|{\bf z}\|\leq 1}|{\rm dist}({\bf z},K_1)- {\rm dist}({\bf z},K_2)|,
$$
where $K_1,K_2\in C(\mathbb{R}^n)$ and ${\rm dist}({\bf z},K) := {\rm inf}_{{\bf u}\in K}\|{\bf z}-{\bf u}\|$ stands for the distance from ${\bf z}$ to $K$. An equivalent
way of defining $\delta$ is
\begin{equation}\label{ddKK}
\delta(K_1,K_2) := {\rm haus}(K_1\cap {\mathbb B}_n,K_2 \cap {\mathbb B}_n),
\end{equation}
where
$$
{\rm haus}(C_1,C_2) := \max\left\{\sup_{{\bf z}\in C_1}{\rm dist}({\bf z},C_2), \sup_{{\bf z}\in C_2}{\rm dist}({\bf z},C_1)\right\}
$$
stands for the Hausdorff distance between the compact sets $C_1,C_2\subset \mathbb{R}^n$ (see \cite[pp.85-86]{Be93}). For more details of the metric $\delta$, see \cite{RW98}. According to Walkup and Wets \cite{WW67}, the
operation $K\mapsto  K^*$ is an isometry on the space $(C(\mathbb{R}^n), \delta)$, that is to say,
$$
\delta(K_1^*,K_2^*)=\delta(K_1,K_2), ~~~{\rm for~all~} K_1,K_2\in C(\mathbb{R}^n).
$$

From Definition \ref{NonsingularDef}, we immediately obtain the following lemma.

\begin{pro}\label{SinglocalP}
Let $(\bar K,\bar{\mathcal{A}})\in C(\mathbb{R}^n)\times {\mathcal T}_{m,n}$. If $\bar{\mathcal{A}}$ is $\bar K$-nonsingular, then there exists a neighborhood  $U$ of $(\bar K,\bar{\mathcal{A}})$ such that $\mathcal{A}$ is $K$-nonsingular for any $(K,\mathcal{A})\in U$.
\end{pro}

\begin{proof}
We prove it by contradiction. Suppose that the conclusion is not true, then there exists a sequence of $\{(K_l,\mathcal{A}^l)\}$ satisfying $(K_l,\mathcal{A}^l)\rightarrow (\bar K,\bar{\mathcal{A}})$ as $l\rightarrow\infty$, such that $\mathcal{A}^l$ is $K_l$-singular for all $l$. Consequently, there exists ${\bf x}^{(l)}\in K_l\cap {\mathbb B}_n$ for every $l$, such that
\begin{equation}\label{barKsing}
 \mathcal{A}^l({\bf x}^{(l)})^{m-1}={\bf 0}~~~~\forall~l=1,2,\ldots.
\end{equation}
Since $\{{\bf x}^{(l)}\}\subseteq {\mathbb B}_n$ and ${\mathbb B}_n$ is compact, without loss of generality, we assume ${\bf x}^{(l)}\rightarrow \bar {\bf x}\in {\mathbb B}_n$ as $l\rightarrow\infty$. Furthermore, since ${\bf x}^{(l)}\in K_l\cap {\mathbb B}_n$  for every $l$ and $\delta(K_l,\bar K)\rightarrow 0$ as $l\rightarrow\infty$, by (\ref{ddKK}), it is not difficult to know that $\bar x\in \bar K$. Accordingly, by letting $l\rightarrow\infty$ in (\ref{barKsing}), it follows from $\mathcal{A}^l\rightarrow \bar{\mathcal{A}}$ and ${\bf x}^{(l)}\rightarrow \bar {\bf x}$ that $\bar {\mathcal{A}}\bar{{\bf x}}^{m-1}={\bf 0}$. It is a contradiction, because $\bar{\mathcal{A}}$ is $\bar K$-nonsingular and $\bar{\bf x}\in \bar K\backslash\{{\bf 0}\}$.
\qed \end{proof}

Hereafter, for given $\mathcal{A}\in {\mathcal T}_{m,n}$ and $K\in C(\mathbb{R}^n)$, we denote
$$
\displaystyle{\rm Tpos}(K,\mathcal{A}):=\left\{\mathcal{A}{\bf x}^{m-1}~|~{\bf x}\in K\right\}.
$$
Clearly, when taking $m=2$ (i.e., $\mathcal{A}$ is a matrix), ${\rm Tpos}({\mathbb{R}_+^n},\mathcal{A})$ reduces to the closed convex cone generated by $\mathcal{A}$ (see \cite{CPS92}). However, when $m\geq 3$, ${\rm Tpos}({\mathbb{R}_+^n},\mathcal{A})$ is not convex in general, but still remains the closedness that will be proved in the following lemma.

\begin{lemma}\label{Closed}
Let $\bar{\mathcal{A}}\in {\mathcal T}_{m,n}$ and $\bar K\in C(\mathbb{R}^n)$. If $\bar{\mathcal{A}}$ is $\bar K$-nonsingular, then ${\rm Tpos}(\bar K,\bar{\mathcal{A}})$ is a closed cone.
\end{lemma}

\begin{proof}
We first prove that ${\rm Tpos}(\bar K,\bar{\mathcal{A}})$ is a cone, i.e., $t{\bf y}\in {\rm Tpos}(\bar K,\bar{\mathcal{A}})$ for any ${\bf y}\in {\rm Tpos}(\bar K,\bar{\mathcal{A}})$ and $t\in \mathbb{R}_+$. Since ${\bf y}\in {\rm Tpos}(\bar K,\bar{\mathcal{A}})$, there exists ${\bf x}\in\bar K$ such that ${\bf y}=\bar{\mathcal{A}}{\bf x}^{m-1}$, which implies that $t{\bf y}=\bar{\mathcal{A}}(t^{\frac{1}{m-1}}{\bf x})^{m-1}\in {\rm Tpos}(\bar K,\bar{\mathcal{A}})$ since $t^{\frac{1}{m-1}}{\bf x}\in \bar K$.

We now prove that ${\rm Tpos}(\bar K,\bar{\mathcal{A}})$ is closed. Note that the function $F$ defined by \eqref{tenx} is a compact map, which implies $F(\bar K\cap {\mathbb B}_n)$ is compact, since $\bar K\cap {\mathbb B}_n$ is a compact basis of $\bar K$. It is clear that $F(\bar K\cap {\mathbb B}_n)\subseteq{\rm Tpos}(\bar K,\bar{\mathcal{A}})$, which implies that ${\rm cone}(F(\bar K\cap {\mathbb B}_n))\subseteq{\rm Tpos}(\bar K,\bar{\mathcal{A}})$, since ${\rm Tpos}(\bar K,\bar{\mathcal{A}})$ is a cone.  Moreover, we claim that ${\rm Tpos}(\bar K,\bar{\mathcal{A}})={\rm cone}(F(\bar K\cap {\mathbb B}_n))$. In fact, for any ${\bf y}\in {\rm Tpos}(\bar K,\bar{\mathcal{A}})$, there exists ${\bf x}\in \bar K$ such that ${\bf y}=\bar{\mathcal{A}}{\bf x}^{m-1}$. If ${\bf x}={\bf 0}$, it is obvious that ${\bf y}={\bf 0}\in {\rm cone}(F(\bar K\cap {\mathbb B}_n))$. Without loss of generality, we assume that ${\bf x}\in \bar K\backslash\{{\bf 0}\}$. Let $\bar {\bf x}:={\bf x}/\|{\bf x}\|$, and then ${\bf x}=\|{\bf x}\|\bar {\bf x}$. Hence ${\bf y}=\|{\bf x}\|^{m-1}\mathcal{A}\bar {\bf x}^{m-1}\in {\rm cone}(F(\bar K\cap {\mathbb B}_n))$. Therefore, ${\rm Tpos}(\bar K,\bar{\mathcal{A}})={\rm cone}(F(\bar K\cap {\mathbb B}_n))$. Since $\bar{\mathcal{A}}$ is $\bar K$-nonsingular, it holds that ${\bf 0}\notin F(\bar K\cap {\mathbb B}_n)$ from ${\bf 0}\notin \bar K\cap{\mathbb B}_n$. We know, from the compactness of $F(\bar K\cap {\mathbb B}_n)$, that ${\rm Tpos}(\bar K,\bar{\mathcal{A}})$ is closed.
\qed\end{proof}

Notice that the $\bar K$-nonsingularity of $\bar {\mathcal{A}}$ is only a sufficient condition for the closedness of ${\rm Tpos}(\bar K,\bar{\mathcal{A}})$, which will be showed in the following example.
\begin{exam}
Let $\mathcal{A}=(a_{i_1i_2i_3})\in {\mathcal T}_{3,2}$, where $a_{112}=a_{121}=a_{221}=a_{212}=1$ and all others $a_{i_1i_2i_3}=0$. It is clear that $\mathcal{A}{\bf x}^2=2x_1x_2 (1,1)^\top$ for any ${\bf x}\in \mathbb{R}^2$, which means that $\mathcal{A}$ is $\mathbb{R}_+^2$-singular, since $\mathcal{A}{\bf x}^2={\bf 0}$ for ${\bf x}=(1,0)^\top$. However, we can see that ${\rm Tpos}(\mathbb{R}_+^2,\mathcal{A})=\{a(1,1)^\top~|~a\geq 0\}$ which is closed.
\end{exam}

In the rest of this section, we regard ${\rm Tpos}(\cdot,\cdot)$ as a set-valued map from $C(\mathbb{R}^n)\times {\mathcal T}_{m,n}$ into the power set of $\mathbb{R}^n$. The following lemma characterizes the closedness of the map ${\rm Tpos}(\cdot,\cdot)$.

\begin{lemma}\label{Mapclosed}
Let $(\bar K,\bar {\mathcal{A}})\in C(\mathbb{R}^n)\times {\mathcal T}_{m,n}$. If $\bar {\mathcal{A}}$ is $\bar K$-nonsingular, then the map ${\rm Tpos}(\cdot,\cdot)$ is closed at $(\bar K,\bar {\mathcal{A}})$.
\end{lemma}

\begin{proof}
Take any sequences $\{(K_l,\mathcal{A}^l)\}\subseteq C(\mathbb{R}^n)\times {\mathcal T}_{m,n}$ and ${\bf y}^{(l)}\in {\rm Tpos}(K_l,\mathcal{A}^l)$ satisfying $(K_l,\mathcal{A}^l)\rightarrow (\bar K,\bar{\mathcal{A}})$ and ${\bf y}^{(l)}\rightarrow\bar {\bf y}$, respectively. To prove the colsedness of ${\rm Tpos}(\cdot,\cdot)$ at $(\bar K,\bar {\mathcal{A}})$, we need to prove $\bar {\bf y}\in {\rm Tpos}(\bar K,\bar{\mathcal{A}})$.

For every $l$, since ${\bf y}^{(l)}\in {\rm Tpos}(K_l,\mathcal{A}^l)$, there exists ${\bf x}^{(l)}\in K_l$ such that ${\bf y}^{(l)}=\mathcal{A}^l({\bf x}^{(l)})^{m-1}$.
Without loss of generality, we assume ${\bf x}^{(l)}\neq {\bf 0}$ for all $l$. Let ${\bf z}^{(l)}:={\bf x}^{(l)}/\|{\bf x}^{(l)}\|$. Then we have
\begin{equation}\label{kkxyK}
{\bf y}^{(l)}=\|{\bf x}^{(l)}\|^{m-1}\mathcal{A}^l({\bf z}^{(l)})^{m-1}.
\end{equation}
It can be easily seen that ${\bf z}^{(l)}\in K_l\cap {\mathbb B}_n$ for every $l$. Since $\delta(K_l,\bar K)\rightarrow 0$ as $l\rightarrow\infty$, by (\ref{ddKK}), we know that ${\rm dist}({\bf z}^{(l)},\bar K\cap {\mathbb B}_n)\rightarrow 0$. Since $\{{\bf z}^{(l)}\}\subseteq {\mathbb B}_n$, without loss of generality, we assume ${\bf z}^{(l)}\rightarrow \bar {\bf z}$ as $l\rightarrow\infty$, which, together with the closedness of $\bar K\cap {\mathbb B}_n$, implies that $\bar {\bf z}\in \bar K\cap {\mathbb B}_n$. Thanks to $\mathcal{A}^l\rightarrow \bar{\mathcal{A}}$ and ${\bf z}^{(l)}\rightarrow \bar {\bf z}$ as $l\rightarrow\infty$, we have $\mathcal{A}^l({\bf z}^{(l)})^{m-1}\rightarrow\bar{\mathcal{A}}\bar {\bf z}^{m-1}$. It follows from the $\bar K$-nonsingularity of $\bar{\mathcal A}$ that $\bar {\mathcal{A}}\bar {\bf z}^{m-1}\neq {\bf 0}$ since $\bar {\bf z} \in \bar K\backslash \{{\bf 0}\}$. Moreover, since ${\bf y}^{(l)}\rightarrow\bar {\bf y}$ and $\mathcal{A}^l({\bf z}^{(l)})^{m-1}\rightarrow\bar{\mathcal{A}}\bar {\bf z}^{m-1}$, by (\ref{kkxyK}), we know that $\|{\bf x}^{(l)}\|\rightarrow \bar t:=(\|\bar {\bf y}\|/\|\bar {\mathcal{A}}\bar {\bf z}^{m-1}\|)^{\frac{1}{m-1}}\in \mathbb{R}_+$. Therefore, by invoking (\ref{kkxyK}) again, we have $\bar {\bf y}=\bar t^{m-1}\bar{\mathcal{A}}\bar {\bf z}^{m-1}=\bar{\mathcal{A}}(\bar t\bar {\bf z})^{m-1}$ with $\bar t\bar {\bf z}\in \bar K$, which means $\bar {\bf y}\in {\rm Tpos}(\bar K,\bar{\mathcal{A}})$. The desired conclusions follow.
\qed\end{proof}

Indeed, the $\bar K$-nonsingularity of $\bar{\mathcal A}$ is a key condition for Lemma \ref{Mapclosed}. Below, we give an example to show that the $\bar K$-nonsingularity of $\bar {\mathcal{A}}$ is necessary to ensure the closedness of the map ${\rm Tpos}$ at $(\bar K,\bar {\mathcal{A}})$, even for a very special case.


\begin{exam}
Let $\bar{\mathcal{A}}=\left[
\begin{array}{cc}
1&-2\\
1&-2
\end{array}
\right]\in {\mathcal T}_{2,2}$. It is clear that $\bar {\mathcal{A}}$ is $\mathbb{R}_+^2$-singular.  To verify the closedness of the map ${\rm Tpos}$ at $({\mathbb{R}_+^2},\bar{\mathcal{A}})$, we consider the sequence $\{\mathcal{A}^l\}$ defined by $$\mathcal{A}^l=\left[
\begin{array}{cccc}
1&&&-2-\frac{1}{l}\\
1&&&-2
\end{array}
\right]\in {\mathcal T}_{2,2},~~\forall~l=1,2,\ldots.$$ It is clear that $\mathcal{A}^l\rightarrow\bar{\mathcal{A}}$ as $l\rightarrow\infty$. Moreover, we know that $(1,2)\in {\rm Tpos}({\mathbb{R}_+^2},\mathcal{A}^l)$ must hold for all $l$, because for every $l$ there exists ${\bf x}^{(l)}:=(2+2l,l)\in \mathbb{R}_+^2$ such that $\mathcal{A}^l{\bf x}^{(l)}=(1,2)$. On the other hand, we see that ${\rm Tpos}({\mathbb{R}_+^2},\bar {\mathcal{A}})=\{a{\bf e}_2~|~a\in \mathbb{R}\}$. It is clear that $(1,2)\not\in {\rm Tpos}({\mathbb{R}_+^2},\bar {\mathcal{A}})$, which means that the map ${\rm Tpos}(\cdot,\cdot)$ is not closed at $({\mathbb{R}_+^2},\bar {\mathcal{A}})$.
\end{exam}

Now we introduce an important concept, which is an extension of the one defined in \cite[Def. 1.3.2]{CPS92}.
\begin{definition}\label{Compten}
Given $\mathcal{A}\in {\mathcal T}_{m,n}$ and $\alpha\subseteq I_n$, we define $\mathcal{C}_{\mathcal{A}}(\alpha)\in {\mathcal T}_{m,n}$ as
\begin{equation}\label{CAalij}
(\mathcal{C}_{\mathcal{A}}(\alpha))_{i_1i_2\ldots i_m}=\left\{
\begin{array}{ll}
-a_{i_1i_2\ldots i_m},&~{\rm if~}i_1\in I_n,i_2,\ldots,i_m\in \alpha,\\
\sigma_{i_1i_2\ldots i_m},&~{\rm if~}i_1,i_2,\ldots,i_m\not\in \alpha,\\
0,&~{\rm otherwise}.
\end{array}
\right.
\end{equation}
$\mathcal{C}_{\mathcal{A}}(\alpha)$ is then called a {\it complementary tensor} of $\mathcal{A}$, where $\sigma_{i_1i_2\ldots i_m}$ is given in the unit tensor ${\mathcal I}$. The associated cone, ${\rm Tpos}(K,\mathcal{C}_\mathcal{A}(\alpha))$, is called a {\it complementary cone}.
\end{definition}

\begin{pro}\label{property1}
Let $\mathcal{A}\in {\mathcal T}_{m,n}$ and $\alpha\subseteq I_n$. If $\alpha=\emptyset$, then $\mathcal{C}_\mathcal{A}(\alpha)=\mathcal{I}$. If $\alpha\neq\emptyset$ and $\mathcal{A}_\alpha$ is $\mathbb{R}_+^{|\alpha|}$-nonsingular, then the complementary tensor $\mathcal{C}_\mathcal{A}(\alpha)$ of $\mathcal{A}$ is $\mathbb{R}_+^n$-nonsingular.
\end{pro}

\begin{proof}
When $\alpha=\emptyset$, by Definition \ref{Compten}, it is obvious that $\mathcal{C}_\mathcal{A}(\alpha)=\mathcal{I}$. We now  consider the case where $\alpha\neq\emptyset$. Take any ${\bf u}:=({\bf u}_\alpha,{\bf u}_{\bar \alpha})\in \mathbb{R}^n_+$, where $\bar \alpha=I_n\backslash \alpha$. It can be easily seen that
\begin{equation}\label{CAu}
\mathcal{C}_\mathcal{A}(\alpha){\bf u}^{m-1}=\left[
\begin{array}{c}
-\mathcal{A}_\alpha({\bf u}_\alpha)^{m-1}\\
-\mathcal{A}_{\bar\alpha\alpha}({\bf u}_\alpha)^{m-1}+({\bf u}_{\bar\alpha})^{[m-1]}
\end{array}
\right].
\end{equation}
If $\mathcal{C}_\mathcal{A}(\alpha){\bf u}^{m-1}={\bf 0}$, then $\mathcal{A}_\alpha({\bf u}_\alpha)^{m-1}={\bf 0}$ by (\ref{CAu}). Consequently, by the $\mathbb{R}_+^{|\alpha|}$-nonsingularity of $\mathcal{A}_\alpha$, it holds that ${\bf u}_\alpha={\bf 0}$. Moreover, it follows from (\ref{CAu}) that ${\bf u}_{\bar \alpha}={\bf 0}$. Therefore, $\mathcal{C}_\mathcal{A}(\alpha)$ is $\mathbb{R}_+^n$-nonsingular.
\qed\end{proof}

\section{Existence of Solutions}\label{Existence}

The properties of solutions of TCPs have been investigated under certain conditions in earlier papers. In this section, we still study along this line and in particular show more interesting results of the TCP on a nonnegative cone, that is, we prove that TCP(${\mathbb R}^n_+,{\bf q},{\mathcal A}$) has a solution under comparatively weaker conditions.

Here, we first make the following assumption.
\begin{assum} \label{KsinAssum}
Let $(K,\mathcal{A})\in C(\mathbb{R}^n)\times{\mathcal T}_{m,n}$. For every nonempty subset $\alpha$ of $I_n$, ${\rm Tpos}(K,C_{\mathcal{A}}(\alpha))$ is closed.
\end{assum}

From Lemma \ref{Closed} and Property \ref{property1}, we know that, if $\mathcal{A}_{\alpha}$ is $\mathbb{R}_+^{|\alpha|}$-nonsingular for the subset $\alpha$ of $I_n$, then ${\rm Tpos}(\mathbb{R}_+^n,C_{\mathcal{A}}(\alpha))$ is closed. Notice that if $\mathcal{A}$ is strictly copositive, then $\mathcal{A}_{\alpha}$ is  $\mathbb{R}_+^{|\alpha|}$-nonsingular for every nonempty subset $\alpha$ of $I_n$. The following example shows that even if $\mathcal{A}\in {\mathcal T}_{m,n}$ is copositive and $\mathcal{A}_{\alpha}$ is  $\mathbb{R}_+^{|\alpha|}$-nonsingular for every nonempty subset $\alpha$ of $I_n$, it is not necessarily strictly copositive.

\begin{exam}
Let $\mathcal{A}=(a_{i_1i_2i_3})\in {\mathcal T}_{3,2}$, where $a_{111}=a_{222}=1$, $a_{112}=a_{122}=-1$ and others are zeros. Then for any ${\bf x}\in \mathbb{R}^2$, we have $\mathcal{A}{\bf x}^2=(x_1^2-x_1x_2-x^2_2,x^2_2)$ and $\mathcal{A}{\bf x}^3=(x_1+x_2)(x_1-x_2)^2$. Consequently, it is easy to verify that $\mathcal{A}_\alpha$ is nonsingular for every nonempty subset $\alpha\subseteq\{1,2\}$ and $\mathcal{A}$ is copositive, but not strictly copositive.
\end{exam}

The above example efficiently shows that the combination of $\mathbb{R}^{|\alpha|}_+$-nonsingularity and copositivity is relatively weaker than the strict copositivity. For given $K\in C(\mathbb{R}^n)$ and $\mathcal{A}\in {\mathcal T}_{m,n}$, we now denote
\[\label{Qdef}Q(K,\mathcal{A}):=\left\{{\bf q}\in \mathbb{R}^n~|~{\rm SOL}(K,{\bf q},\mathcal{A})\neq \emptyset\right\}.\]

\begin{proposition}\label{ClosedQ}
For given $K\in C(\mathbb{R}^n)$ and $\mathcal{A}\in {\mathcal T}_{m,n}$. Then, $Q(K,\mathcal{A})$ given by \eqref{Qdef} is a cone. In particular, if ${\mathcal A}$ satisfies Assumption \ref{KsinAssum} with $K=\mathbb{R}_+^n$, then $Q(\mathbb{R}_+^n,\mathcal{A})$ is a closed cone.
\end{proposition}

\begin{proof}
We first prove that $Q(K,\mathcal{A})$ is a cone, i.e., ${\bf q}\in Q(K,\mathcal{A})$ and $t\in \mathbb{R}_+$ implies $t{\bf q}\in Q(K,\mathcal{A})$. When $t=0$ or ${\bf q}={\bf 0}$, the conclusion is obvious since $t{\bf q}={\bf 0}\in Q(K,\mathcal{A})$. We now assume ${\bf q}\neq {\bf 0}$ and $t>0$. Since ${\bf q}\in Q(K,\mathcal{A})$, we know that ${\rm SOL}(K,{\bf q},\mathcal{A})\neq \emptyset$. Take ${\bf x}\in{\rm SOL}(K,{\bf q},\mathcal{A})$. Then, it is easy to see that $t^{\frac{1}{m-1}}{\bf x}\in{\rm SOL}(K,t{\bf q},\mathcal{A})$, and hence $t{\bf q}\in Q(K,\mathcal{A})$.

We now prove the closedness of $Q(\mathbb{R}_+^n,\mathcal{A})$. To this end, we first prove
\begin{equation}\label{QACpos}Q(\mathbb{R}_+^n,\mathcal{A})=\bigcup_{\alpha\subseteq I_n}{\rm Tpos}({\mathbb{R}_+^n},C_{\mathcal{A}}(\alpha)).
 \end{equation}
 For any ${\bf q}\in {\rm Tpos}({\mathbb{R}_+^n},C_{\mathcal{A}}(\alpha))$ with some $\alpha\subseteq I_n$, there exists ${\bf u}=({\bf u}_\alpha,{\bf u}_{\bar \alpha})\in \mathbb{R}_+^n$ such that ${\bf q}=C_{\mathcal{A}}(\alpha){\bf u}^{m-1}$, that is,
$$
\left\{
\begin{array}{l}
{\bf q}_{\alpha}=-\mathcal{A}_{\alpha}({\bf u}_\alpha)^{m-1},\\
{\bf q}_{\bar\alpha}=-\mathcal{A}_{\bar\alpha\alpha}({\bf u}_\alpha)^{m-1}+({\bf u}_{\bar\alpha})^{[m-1]},\\
\end{array}
\right.
$$
and hence
$$
\left\{
\begin{array}{l}
\mathcal{A}_{\alpha}({\bf u}_\alpha)^{m-1}+{\bf q}_{\alpha}={\bf 0},\\
\mathcal{A}_{\bar\alpha\alpha}({\bf u}_\alpha)^{m-1}+{\bf q}_{\bar\alpha}=({\bf u}_{\bar\alpha})^{[m-1]}.
\end{array}
\right.
$$
This means that ${\bf x}=({\bf u}_\alpha,{\bf 0})\in {\rm SOL}(\mathbb{R}_+^n,{\bf q},\mathcal{A})$, and hence ${\bf q}\in Q(\mathbb{R}_+^n,\mathcal{A})$. Therefore, we have
$$\bigcup_{\alpha\subseteq I_n}{\rm Tpos}({\mathbb{R}_+^n},C_{\mathcal{A}}(\alpha))\subseteq Q(\mathbb{R}_+^n,\mathcal{A}).$$
Conversely, for every ${\bf q}\in Q(\mathbb{R}_+^n,\mathcal{A})$, there exists ${\bf x}\in\mathbb{R}^n$ such that
\begin{equation}\label{RREt}
{\bf x}\geq {\bf 0},~\mathcal{A}{\bf x}^{m-1}+{\bf q}\geq {\bf 0}~~{\rm and}~~{\bf x}^\top (\mathcal{A}{\bf x}^{m-1}+{\bf q})=0.
\end{equation}
If ${\bf x=0}$, then ${\bf q}\geq {\bf 0}$ by (\ref{RREt}). Consequently, we have ${\bf q}\in{\rm Tpos}({\mathbb{R}_+^n},C_{\mathcal{A}}(\emptyset))$ since $C_{\mathcal{A}}(\emptyset)=\mathcal{I}$ by Property \ref{property1}. We consider the case of ${\bf x\neq 0}$. Denote $\alpha=\{i\in I_n~|~x_i>0\}$ and $\bar\alpha=I_n\backslash \alpha$. It then follows from (\ref{RREt}) that
$$
\left\{\begin{array}{l}
\mathcal{A}_\alpha({\bf x}_\alpha)^{m-1}+{\bf q}_\alpha={\bf 0},\\
{\bf y}_{\bar \alpha}:=\mathcal{A}_{\bar \alpha\alpha}({\bf x}_\alpha)^{m-1}+{\bf q}_{\bar\alpha}\geq{\bf 0},
\end{array}
\right.
$$
which implies
\begin{equation}\label{RREt1}
\left\{\begin{array}{l}
{\bf q}_\alpha=-\mathcal{A}_\alpha({\bf x}_\alpha)^{m-1},\\
{\bf q}_{\bar\alpha}=-\mathcal{A}_{\bar \alpha\alpha}({\bf x}_\alpha)^{m-1}+{\bf y}_{\bar \alpha}.
\end{array}
\right.
\end{equation}
For such $\alpha\subseteq I_n$, we define $C_{\mathcal{A}}(\alpha)$ by (\ref{CAalij}) and let ${\bf u}=({\bf x}_\alpha, ({\bf y}_{\bar \alpha})^{[\frac{1}{m-1}]})$. It is clear that ${\bf u}\in \mathbb{R}_+^n$. Moreover, it holds that ${\bf q}= C_{\mathcal{A}}(\alpha){\bf u}^{m-1}$, which implies that there exists a subset $\alpha$ of $I_n$ satisfying ${\bf q}\in {\rm Tpos}({\mathbb{R}_+^n},C_{\mathcal{A}}(\alpha))$. Therefore, (\ref{QACpos}) holds. Moreover, for a tensor $\mathcal{A}\in {\mathcal T}_{m,n}$, there are $2^n$ (not necessarily all distinct) ${\rm Tpos}(\mathbb{R}_+^n,C_{\mathcal{A}}(\alpha))$, which, together with Assumption \ref{KsinAssum} with $K=\mathbb{R}_+^n$, implies that $Q(\mathbb{R}_+^n,\mathcal{A})$ is closed.
\qed\end{proof}

\begin{proposition}\label{prop1}
Let ${\bf q}\in \mathbb{R}^n$ and $\mathcal{A}\in{\mathcal T}_{m,n}$. If ${\mathcal A}$ satisfies Assumption \ref{KsinAssum} with $K=\mathbb{R}_+^n$ and ${\rm SOL}(\mathbb{R}_+^n, {\bf q},\mathcal{A})=\emptyset$, then there exists a neighborhood $V$ of ${\bf q}$ such that ${\rm SOL}(\mathbb{R}_+^n,{\bf q}^\prime,\mathcal{A})=\emptyset$ for any ${\bf q}^\prime\in V$.
\end{proposition}
\begin{proof}
Since ${\rm SOL}(\mathbb{R}_+^n,{\bf q},\mathcal{A})=\emptyset$, it holds that ${\bf q}\not\in Q(\mathbb{R}_+^n,\mathcal{A})$. From Proposition \ref{ClosedQ}, we know $Q(\mathbb{R}_+^n,\mathcal{A})$ is closed, which means there exists  a neighborhood $V$ of ${\bf q}$, such that $V\cap Q(\mathbb{R}_+^n,\mathcal{A})=\emptyset$, and hence the desired result follows.
\qed\end{proof}

Let $\mathcal{A}=(a_{i_1i_2\ldots i_m})\in{\mathcal T}_{m,n}$. It is easy to see that ${\rm SOL}(\mathbb{R}_+^n,{\bf 0},\mathcal{A})\neq\emptyset$. If TCP($\mathbb{R}_+^{n},{\bf q}, {\mathcal{A}}$) always has a solution for any ${\bf q}\in{\mathbb R}^n$, then we call ${\mathcal A}$ a Q-tensor. It has been well documented in \cite{BHW16,CQW16,DLQ15,GLQX15,SQ14,WHB16} that strictly semi-positive tensors, P-tensors, strictly copositive tensors, ER-tensors, positive tensors, and R-tensors are Q-tensors. If ${\mathcal A}$ is a strong M-tensor satisfying $a_{i_1\cdots i_m}=0$ whenever $i_j\neq i_k$ for some $j\neq k$, Gowda et al. \cite{GLQX15} proved that ${\mathcal A}$ is also a Q-tensor. If a symmetric tensor is copositive, then such a tensor is semi-positive. Correspondingly, TCP($\mathbb{R}_+^{n},{\bf q}, {\mathcal{A}}$) with a semi-positive ${\mathcal A}$ has a unique solution for all ${\bf q}>0$ (see \cite{SQ16}). When the underlying ${\mathcal A}$ is a semi-positive R$_0$-tensor, the authors of \cite{SQ14} proved that ${\mathcal A}$ must be an R-tensor. Consequently, TCP($\mathbb{R}_+^{n},{\bf q}, {\mathcal{A}}$) has a solution for any ${\bf q}\in{\mathbb R}^n$. In addition, it has been proved in \cite{HSW15} that a nonnegative tensor is also an R-tensor. When we consider a special case of TCP($\mathbb{R}_+^{n},{\bf q}, {\mathcal{A}}$) with a second-order tensor (i.e., ${\mathcal A}$ is a matrix), it is known from \cite{CPS92} that if ${\mathcal A}$ is copositive, then, for all ${\bf q}\in{\mathbb R}^n$ with the following property
$$\left[\;{\bf v}\geq 0,\; {\mathcal A}{\bf v}\geq 0,\;{\bf v}^\top {\mathcal A}{\bf v}=0 \;\right]\;\Rightarrow \;{\bf v}^\top {\bf q}\geq 0,$$
TCP($\mathbb{R}_+^{n},{\bf q}, {\mathcal{A}}$) (which is actually an LCP) has a solution. Inspired by such a result, we are further concerned with the solution existence of TCPs ($m\geq 3$) with copositive tensors.

To simplify the notations, let $S_{\mathcal{A}}:={\rm SOL}(\mathbb{R}_+^n,{\bf 0},\mathcal{A})$. It is clear that $S_{\mathcal{A}}$ is a cone and $S_{\mathcal{A}}\subseteq \mathbb{R}_+^n$. Consequently, it holds that $\mathbb{R}_+^n\subseteq S_{\mathcal{A}}^*$, which implies that ${\rm int}(S_{\mathcal{A}}^*)$ contains $\mathbb{R}^n_{++}$. For a class of multivalued functions with the ``upper limiting homogeneity" (ULH) property, Gowda and Pang \cite{GP92} derived the existence of a solution to the corresponding multivalued complementarity problems. According to their result, we have the following theorem.
\begin{theorem}\label{CopExistence}
Let $\bar{\mathcal{A}}\in {\mathcal T}_{m,n}$ be given. Suppose that $\bar{\mathcal{A}}$ is copositive and satisfies Assumption \ref{KsinAssum} with $K=\mathbb{R}_+^n$. Then ${\rm TCP}(\mathbb{R}_+^n,{\bf q},\bar{\mathcal{A}})$ has a solution for any ${\bf q}\in S_{\bar{\mathcal{A}}}^*$. Moreover, ${\rm SOL}(\mathbb{R}_+^n,{\bf q},\bar{\mathcal{A}})$ is compact for any ${\bf q}\in {\rm int}(S_{\bar{\mathcal{A}}}^*)$
\end{theorem}

\begin{proof}
We first know that $Q(\mathbb{R}_+^n,\bar{\mathcal{A}})$ is closed, by Proposition \ref{ClosedQ}. We now take the multivalued function $\Phi({\bf x})$ in \cite{GP92} as $\Phi({\bf x}):=F({\bf x}):=\bar {\mathcal{A}}{\bf x}^{m-1}$, then $\Phi$ has the ULH property with degree $m-1$. Moreover, the function $\Gamma$, defined by (3) in \cite{GP92}, is equal to $F(\bf x)$. Consequently, under the condition that $\bar{{\mathcal A}}$ is copositive, by Theorem 2 in \cite{GP92}, we know that ${\rm SOL}(\mathbb{R}_+^n,{\bf q},\bar{\mathcal A})$ is a nonempty compact set for every ${\bf q}\in {\rm int}(S_{\bar{\mathcal A}}^*)$, which implies ${\rm int}(S_{\bar{\mathcal A}}^*)\subset Q(\mathbb{R}_+^n,\bar{\mathcal{A}})$. This implies, together with the closedness of $Q(\mathbb{R}_+^n,\bar{\mathcal{A}})$, that $S_{\bar{\mathcal A}}^*\subseteq Q(\mathbb{R}_+^n,\bar{\mathcal{A}})$. We obtain the desired results and complete the proof.
\qed\end{proof}

%


\begin{remark}
For given $\bar{\mathcal{A}}\in {\mathcal T}_{m,n}$, if for every subset $\alpha$ of $I_n$, $\bar{\mathcal{A}}_{\alpha}$ is $\mathbb{R}_+^{|\alpha|}$-nonsingular, then Assumption \ref{KsinAssum} holds by Property \ref{property1} and Lemma \ref{Mapclosed}, and it is easy to see that $S_{\bar{\mathcal{A}}}=\{0\}$, and hence $S_{\bar{\mathcal{A}}}^*=\mathbb{R}^n$, which means that ${\rm TCP}(\mathbb{R}_+^n,{\bf q},\bar{\mathcal{A}})$ has a solution for every ${\bf q}\in \mathbb{R}^n$. A similar result was presented by Gowda and Pang, see  \cite[Corollary 2]{GP92}.
\end{remark}

Now, we state the final result of this section. First, we introduce an important definition that will be used in our result. For given two tensors $\mathcal{A}=(a_{i_1i_2\ldots i_m})$ and $\mathcal{B}=(b_{i_1i_2\ldots i_m})\in {\mathcal T}_{m,n}$, the distance between $\mathcal{A}$ and $\mathcal{B}$ is measured by means of
the expression
$$
\left\|\mathcal{A}-\mathcal{B}\right\|_F:=\sqrt{\sum_{1\leq i_1,\ldots,i_m\leq n}\left(a_{i_1i_2\ldots i_m}-b_{i_1i_2\ldots i_m}\right)^2}.
$$

\begin{theorem}
Let $\bar{\mathcal{A}}\in {\mathcal T}_{m,n}$ and $\bar {\bf q}\in {\rm int} (S_{\bar{\mathcal{A}}}^{*})$. Suppose there exists a neighborhood $\bar {\mathcal{N}}$ of $\bar{\mathcal{A}}$ such that every $\mathcal{A}\in \bar{\mathcal{N}}$ satisfies Assumption \ref{KsinAssum}. Then, there exist positive scalars $ \varepsilon$ and $c$ such that for all $({\bf q},\mathcal{A})$ with $\|{\bf q}-\bar{\bf q}\|+\|\mathcal{A}-\bar{\mathcal{A}}\|_F\leq\varepsilon$ and ${\mathcal{A}} $ being copositive, the following statements hold:
\begin{itemize}
\itemindent 6pt
\item[{\rm (i)}] the ${\rm TCP}(\mathbb{R}_+^n,{\bf q},\mathcal{A})$ is solvable;
\item[{\rm (ii)}] for all $ {\bf x}\in {\rm SOL}(\mathbb{R}_+^n,{\bf q},\mathcal{A})$, it hold that $\|{\bf x}\|\leq c. $
\end{itemize}
\end{theorem}	

\begin{proof} We first prove part (i) that, there exists a neighborhood $\mathcal{N}$ of $(\bar {\bf q},\bar {\mathcal{A}})$, such that ${\bf q}\in {\rm int}(S_{\mathcal{A}}^*)$ for any $({\bf q},\mathcal{A})\in \mathcal{N}$, by contradiction. Suppose such no neighborhood exist, then there exists a sequences $\{{\bf q}^{(l)},\mathcal{A}^{l},{\bf v}^{(l)}\} \subset \mathbb{R}^{n}\times {\mathcal T}_{m,n}\times \mathbb{R}^{n}$ satisfying $({\bf q}^{(l)},\mathcal{A}^l)\rightarrow (\bar {\bf q},\bar{\mathcal{A}})$ and ${\bf v}^{(l)}\in {\rm SOL}(\mathbb{R}^n_+,{\bf 0},\mathcal{A}^{l})$, such that $({\bf v}^{(l)})^\top {\bf q}^{(l)}\leq 0$. Without loss of generality, we assume $\|{\bf v}^{(l)}\|=1$ for each $l$. Let $ \tilde{\bf v} $ be an accumulation point of the sequence $\{{\bf v}^{(l)}\}$. It is easy to see that $ \tilde{\bf v} $ is a nonzero solution of ${\rm TCP}(\mathbb{R}_+^n,{\bf 0},\bar {\mathcal A})$ and satisfies $\tilde{\bf v}^{\top}{\bar{\bf q}}\leq0$, which contradicts the assumption ${\bar{\bf q}}\in {\rm int}(S_{\bar{\mathcal A}}^{*})$. Hence, the existence of $\mathcal{N}$ follows. Moreover, without loss of generality, we may assume that $\mathcal{A}$ satisfies Assumption \ref{KsinAssum} for any $({\bf q},\mathcal{A})\in \mathcal{N}$. Consequently, by Theorem \ref{CopExistence}, we know that, for any $({\bf q},\mathcal{A}) \in \mathcal{N}$ with $\mathcal{A}$ being copositive, the ${\rm TCP}(\mathbb{R}_+^n,{\bf q},\mathcal{A})$ must have a solution. Hence part (i) is established.

We prove part (ii) by contradiction. Suppose that no such a constant $c$ exists. Then there exist sequences $\{({\bf q}^{(l)},\mathcal{A}^l)\}$ with $({\bf q}^{(l)},\mathcal{A}^l) \longrightarrow (\bar {\bf q},\bar{\mathcal{A}})$ and $\{{\bf x}^l\}$ with $\|{\bf x}^{l}\|\longrightarrow\infty$, such that for each $l$,  $\mathcal{A}^{l}$ is copositive, and
\begin{equation}\label{THJU}
{\bf x}^{(l)}\geq{\bf 0},~~{\bf w}^{(l)}:={\bf q}^{(l)}+\mathcal{A}^{l}({\bf x}^{(l)}) ^{m-1}\geq {\bf 0}~~{\rm and}~~({\bf x}^{(l)})^{\top}{\bf w}^{(l)}=0.
\end{equation}
Let $\bar {\bf x}$ be a subsequential limit of normalized sequence $\{{\bf x}^{(l)}/\|{\bf x}^{(l)}\|\}$. It holds for \eqref{THJU} that
$$\left\{\begin{array}{l}
{\bf x}^{(l)}/\|{\bf x}^{(l)}\|\geq {\bf 0},\\
\frac{{\bf q}^{(l)}}{\|{\bf x}^{(l)}\|^{m-1}}+\mathcal{A}^{l}\left(\frac{{\bf x}^{(l)}}{\|{\bf x}^{(l)}\|} \right)^{m-1} \geq{\bf 0},\\
0=\frac{({\bf q}^{(l)})^{\top}}{\|{\bf x}^{(l)}\|^{m-1}} \frac{{\bf x}^{(l)}}{\|{\bf x}^{(l)}\|}+\mathcal{A}^{l}\left(\frac{{\bf x}^{(l)}}{\|{\bf x}^{(l)}\|} \right)^{m},
\end{array}\right.$$
which, by passing to the limit $ l\longrightarrow\infty $ in these expressions, implies that $\bar {\bf x}\in S_{\bar{\mathcal{A}}}\backslash\{{\bf 0}\}$.
On the other hand, the copositivity of $\mathcal{A}^l$ implies that
$$
0=\frac{( {\bf q}^{(l)}) ^{\top}{\bf x}^{(l)}}{\|{\bf x}^{(l)}\|}+\frac{({\bf x}^{(l)})^{\top} \mathcal{A}^{l}({\bf x}^{(l)})^{m-1} }{\|{\bf x}^{(l)}\|}\geq({\bf q}^{(l)})^{\top} \frac{{\bf x}^{(l)}}{\|{\bf x}^{(l)}\|},
$$
By letting $l\longrightarrow\infty$, we obtain $\bar {\bf q}^\top \bar {\bf x}\leq 0$, which contradicts $\bar {\bf q}\in {\rm int}(S_{\bar{\mathcal{A}}}^*)$. This contradiction completes the proof of part (ii).
\qed\end{proof}

\section{Topological Properties and Stability}\label{sectop}

As far as we know, properties on solution set such as topological properties and stability are discussed much less in TCP literature. Thus, we in this section investigate these results to enrich the theory of TCPs.

\subsection{Topological properties}\label{subsectop}

We assume no special structure on $K$ other than the fact that it is closed and convex.
First, we introduce a concept for the study of topological properties of the solution set. As far as the semicontinuity concepts are concerned, we use the following terminology (see \cite[Section 6.2]{Be93}).

\begin{definition}\label{def51}
Let $W$ and $Y$ be two topological spaces. The mapping $\Psi:\;W\rightarrow 2^Y$ is said to be upper-semicontinuous at $\bar w\in W$, if whenever $U$ is an open subset of $Y$ containing $\Psi(\bar w)$, then $U$ contains $\Psi(w)$ for each $w$ in some neighborhood of $\bar w$.
\end{definition}

Here, we define the solution set of TCPs as the mapping ${\rm SOL}(\cdot,\cdot,\cdot): C({\mathbb R}^n)\times{\mathbb R}^n\times{\mathcal T}_{m,n}\rightarrow 2^{\mathbb{R}}$. Then, we have the following statements, which are closely related to the results presented in \cite{LHQ16} for tensor eigenvalue complementarity problems.

\begin{theorem}\label{theorem1} The following three statements are true
\begin{itemize}
\itemindent 6pt
\item[{\rm(i)}] The set $ \Sigma:= \left\{(K,{\bf q},\mathcal{A}, {\bf x}) \in {C}( {\mathbb R}^{n}) \times \mathbb{R}^n\times {\mathcal T}_{m,n}\times \mathbb{R}^{n}~|~{\bf x}\in{\rm SOL}(K,{\bf q},\mathcal{A})\right\}$ is closed in the product space ${C}( {\mathbb R}^{n}) \times \mathbb{R}^n\times {\mathcal T}_{m,n}\times \mathbb{R}^{n}$.
\item[{\rm(ii)}]  Let  $(\bar K,\bar {\bf q},\bar{\mathcal{A}}) \in {C}( {\mathbb R}^{n}) \times \mathbb{R}^n\times {\mathcal T}_{m,n}$. If $\bar{\mathcal{A}}$ is $\bar K$-regular, then the following set $$\bigcup_{(K,{\bf q},\mathcal{A}) \in \mathcal{N}}{\rm SOL}(K,{\bf q},\mathcal{A})$$ is bounded for some neighborhood $\mathcal{N}$ of $(\bar K,\bar {\bf q},\bar{\mathcal{A}})$.
\item[{\rm(iii)}]  Let  $(\bar K,\bar {\bf q},\bar{\mathcal{A}}) \in {C}( {\mathbb R}^{n}) \times \mathbb{R}^n\times {\mathcal T}_{m,n}$. If $\bar{\mathcal{A}}$ is $\bar K$-regular, then the map ${\rm SOL}(\cdot,\cdot,\cdot)$ is upper-semicontinuous at $(\bar K,\bar {\bf q},\bar{\mathcal{A}})$.
\end{itemize}
\end{theorem}

\begin{proof} (i) The closedness of $\Sigma$ at $(\bar K,\bar {\bf q},\bar{\mathcal{A}}) \in {C}( {\mathbb R}^{n}) \times \mathbb{R}^n\times {\mathcal T}_{m,n}$ amounts to saying that
\begin{equation}
\left.
\begin{array}{r}
(K_l, {\bf q}^{(l)}, \mathcal{A}^{l})\rightarrow (\bar K, \bar {\bf q}, \bar{\mathcal{A}}),~{\bf x}^{(l)}\rightarrow \bar {\bf x}\\
{\bf x}^{(l)}\in{\rm SOL}(K_l,{\bf q}^{(l)},\mathcal{A}^l)
\end{array}
\right\}\Rightarrow \bar {\bf x}\in{\rm SOL}(\bar K,\bar {\bf q},\bar {\mathcal{A}}),
\end{equation}
where $(K_l,{\bf q}^{(l)}, \mathcal{A}^{l})\in{C}( {\mathbb R}^{n}) \times \mathbb{R}^n\times {\mathcal T}_{m,n}$. Since ${\bf x}^{(l)}\in{\rm SOL}(K_l,{\bf q}^{(l)},\mathcal{A}^l)$, we have
\begin{equation}\label{Knuxnu}
{\bf x}^{(l)}\in K_l,~{\bf w}^{(l)}:=\mathcal{A}^l({\bf x}^{(l)})^{m-1}+{\bf q}^{(l)}\in K_l^* ~{\rm and}~\langle {\bf x}^{(l)},{\bf w}^{(l)}\rangle=0.
\end{equation}
Note that $\delta(K_{1}^{*},K_{2}^{*}) =\delta(K_{1},K_{2})$ for any $K_{1},K_{2}\in C(\mathbb{R}^{n})$. By passing to the limit in (\ref{Knuxnu}), we then have
$$\bar {\bf x}\in \bar K,~\bar {\bf w}:=\mathcal{A}\bar {\bf x}^{m-1}+\bar {\bf q}\in \bar K^* ~{\rm and}~\langle \bar {\bf x},\bar {\bf w}\rangle=0,
$$
which implies that $\bar {\bf x}\in{\rm SOL}(\bar K,\bar {\bf q},\bar{\mathcal{A}})$, and arrive at the desired conclusion part (i).

We argue part (ii) by contradiction. Suppose that the conclusion is not true, then there exists a sequence $\{(K_l, {\bf q}^{(l)}, \mathcal{A}^{l})\}$ satisfying
$$\delta(K_l,\bar K)\rightarrow 0, ~\|{\bf q}^{(l)}-\bar {\bf q}\|\rightarrow 0, ~\|\mathcal{A}^{l}-\bar{\mathcal{A}}\|_F\rightarrow 0~~{\rm and}~~\|{\bf x}^{(l)}\|\rightarrow\infty,$$
such that ${\bf x}^{(l)}\in{\rm SOL}(K_l,{\bf q}^{(l)},\mathcal{A}^l)$, i.e., (\ref{Knuxnu}) hold for any $l=1,2,\cdots$. Here, $\bar{\mathcal{A}}\in {\mathcal T}_{m,n}$. Consequently, it holds that
\begin{equation}\label{KnuLamnu}\mathcal{A}^l({\bf x}^{(l)})^{m}+({\bf q}^{(l)})^\top {\bf x}^{(l)}=0.\end{equation}
Let $\bar {\bf x}^{(l)}={\bf x}^{(l)}/\|{\bf x}^{(l)}\|$. We then assume that $\bar {\bf x}^{(l)}\rightarrow \bar {\bf x}$ with $\|\bar {\bf x}\|=1$ due to $\|\bar {\bf x}^{(l)}\|=1$. From ${\bf x}^{(l)}/\|{\bf x}^{(l)}\|\in K_l$ and $\delta(K_l,\bar K)\rightarrow 0$, we obtain $\bar{\bf x}\in \bar K\backslash \{{\bf 0}\}$. It follows from (\ref{KnuLamnu}) that
$$\mathcal{A}^l(\bar {\bf x}^{(l)})^{m}+\frac{1}{\|{\bf x}^{(l)}\|^{m-1}}({\bf q}^{(l)})^\top \bar {\bf x}^{(l)}=0.$$
Moreover, by passing the limit in the expression above, we have
$\bar{\mathcal{A}}\bar{\bf x}^{m}=0$, which contradicts the condition that $\bar{\mathcal{A}}$ is $\bar K$-regular, since $\bar {\bf x}\in \bar K\backslash\{{\bf 0}\}$.

To show part (iii), i.e., the upper-semicontinuity of ${\rm SOL}(\cdot,\cdot,\cdot)$ at $(\bar K,\bar {\bf q},\bar{\mathcal{A}})$, we also argue it by contradiction. Suppose that ${\rm SOL}(\cdot,\cdot,\cdot)$ is not upper-semicontinuous at $(\bar K,\bar {\bf q},\bar{\mathcal{A}})$, then we could find an open set $\overline{U}\subset \mathbb{R}^{n} $ and a sequence $ \{(K_l,{\bf q}^{(l)},\mathcal{A}^l)\}$ satisfying $(K_l,{\bf q}^{(l)},\mathcal{A}^l)\rightarrow (\bar K,\bar {\bf q},\bar{\mathcal{A}})$, such that $$ {\rm SOL}(\bar K,\bar {\bf q},\bar{\mathcal{A}})\subset\bar{U},~~{\rm but}~~{\rm SOL}(K_l,{\bf q}^{(l)},\mathcal{A}^l)\cap (\mathbb{R}^{n}\backslash\bar{U})\neq \emptyset$$
 for any $l=1,2,\cdots$.
 Now, for each $l$, taking ${\bf x}^{(l)}\in {\rm SOL}(K_l,{\bf q}^{(l)},\mathcal{A}^l)\cap (\mathbb{R}^{n}\backslash\bar{U})$, it then follows by part (ii) that the sequence $\{{\bf x}^{(l)}\}$ admits a converging subsequence. By part (i), the corresponding limit must be in $ {\rm SOL}(\bar K,\bar {\bf q},\bar {\mathcal{A}})\cap {\mathbb R}^{n} \backslash\bar U$, which, together with the fact that ${\rm SOL}(\bar K,\bar {\bf q},\bar{\mathcal{A}})\subset\bar{U}$,  leads to a contradiction.
\qed\end{proof}

From part (i) of Theorem \ref{theorem1}, we know that, for any $(\bar K,\bar {\bf q},\bar{\mathcal{A}}) \in {C}( {\mathbb R}^{n}) \times \mathbb{R}^n\times {\mathcal T}_{m,n}$, ${\rm SOL}(\bar K,\bar {\bf q},\bar{\mathcal{A}})$ is a closed set in $\mathbb{R}^n$; By Theorem \ref{theorem1} (ii), we know that, for any $(\bar K,\bar {\bf q},\bar{\mathcal{A}}) \in {C}(\mathbb{R}^{n}) \times \mathbb{R}^n\times {\mathcal T}_{m,n}$ with $\bar{\mathcal{A}}$ being $\bar K$-regular, ${\rm SOL}(\bar K,\bar {\bf q},\bar{\mathcal{A}})$ is bounded; Moreover, for any $\varepsilon>0$, by Theorem \ref{theorem1} (i) and (iii), there exists a neighborhood $\mathcal{N}$ of $(\bar K,\bar {\bf q},\bar{\mathcal{A}})$, such that
$$
{\rm SOL}(K,{\bf q},\mathcal{A})\subseteq{\rm SOL}(\bar K,\bar {\bf q},\bar{\mathcal{A}})+\varepsilon {\mathbb B}_n,~~~\forall~(K,{\bf q},\mathcal{A})\in \mathcal{N}.
$$

\subsection{Stability analysis} \label{Stabana}

In this subsection, we study the sensitivity of TCPs at an isolated solution. Consider the ${\rm TCP}(K,\bar {\bf q},\bar{\mathcal{A}})$ with a given solution $\bar {\bf x}$ which is assumed to be locally unique, i.e., there exists a neighborhood $V$ of $\bar{\bf x}$ such that ${\rm SOL}(K, \bar{\bf q}, \bar{\mathcal A})\cap V =\{\bar {\bf x}\}$. The cone $K\in C(\mathbb{R}^n)$ is fixed throughout the rest of this section. Like Section \ref{subsectop}, we also assume no special structure on $K$ other than the fact that it is closed
and convex. We want to study the change of $\bar {\bf x}$ as $(\bar {\bf q},\bar{\mathcal{A}})$ is perturbed. A central question is of course whether the perturbed ${\rm TCP}(K,{\bf q},\mathcal{A})$, where $({\bf q},\mathcal{A})$ is a small perturbation of $(\bar {\bf q},\bar{\mathcal{A}})$, will have a solution that is near $\bar {\bf x}$. To answer this question, we first present the following proposition.

\begin{proposition}\label{Convergsolution}
Let $K\in C(\mathbb{R}^n)$, $\mathcal{A}\in {\mathcal T}_{m,n}$ and ${\bf q}\in \mathbb{R}^n$. If $\bar {\bf x}\in {\rm SOL}(K,{\bf q},\mathcal{A})$ is isolated,
then for every neighborhood $\mathcal{N}$ of $\bar {\bf x}$ satisfying
\begin{equation}\label{Soldsolution}
{\rm SOL}(K,{\bf q},\mathcal{A})\cap {\rm cl}(\mathcal{N}) = \{ \bar {\bf x}\}
\end{equation}
and for every sequence of $\{(\mathcal{A}^l,{\bf q}^{(l)})\}$ converging to $(\mathcal{A},{\bf q})$, every sequence of vectors $\{{\bf x}^{(l)}\}$, where ${\bf x}^{(l)}\in {\rm SOL}(K,{\bf q}^{(l)},\mathcal{A}^l)\cap \mathcal{N}$ for every $l$, converges to $\bar {\bf x}$.
\end{proposition}

\begin{proof}
Note that such a sequence $\{{\bf x}^{(l)}\}$ in the proposition must be bounded. Let ${\bf x}^*$ be any accumulation point of this sequence, which means that there exists a subsequence $\{{\bf x}^{(l_i)}\}$ of $\{{\bf x}^{(l)}\}$ such that ${\bf x}^{(l_i)}\rightarrow {\bf x}^*$ as $i\rightarrow\infty$. Since ${\bf x}^{(l_i)}\in {\rm SOL}(K,{\bf q}^{(l_i)},\mathcal{A}^{l_i})\cap \mathcal{N}$ and $\{(\mathcal{A}^{l_i},{\bf q}^{(l_i)})\}$ converges to $(\mathcal{A},{\bf q})$, by a simple limiting argument, it is easy to know that ${\bf x}^*\in {\rm SOL}(K,{\bf q},\mathcal{A})\cap {\rm cl}(\mathcal{N})$. By (\ref{Soldsolution}), we know that ${\bf x}^*=\bar {\bf x}$. Therefore,
$\{{\bf x}^{(l)}\}$ converges to $\bar {\bf x}$ as $l\rightarrow\infty$.
\qed\end{proof}

Recall that the tangent cone of $K$ at a vector $\bar {\bf x}\in K$, denoted by ${\mathscr T}(\bar {\bf x},K)$, consists of all vectors ${\bf v}\in \mathbb{R}^n$, for which there exists a sequence of vectors $\{{\bf x}^{(k)}\}\subset K$ and a sequence of positive scalars $\{\tau_k\}$ such that
$$
\lim_{k\rightarrow \infty}{\bf x}^{(k)}=\bar {\bf x},~~\lim_{k\rightarrow \infty}\tau_k=0~~{\rm and}~~\lim_{k\rightarrow \infty}\frac{{\bf x}^{(k)}-\bar {\bf x}}{\tau_k}={\bf v}.
$$
For given $\bar {\bf x}\in \mathbb{R}^n$, denote
$$
{\mathbb S}(\bar {\bf x}):=\left\{{\bf v}\in \mathbb{R}^n~|~{\bf v}^\top \left(\mathcal{A}\bar {\bf x}^{m-1}+ {\bf q}\right)=0\right\}.
$$

\begin{proposition}\label{localunique}
Let $K\in C(\mathbb{R}^n)$, $\mathcal{A}\in {\mathcal T}_{m,n}$ and ${\bf q}\in \mathbb{R}^n$. Suppose that $\mathcal{A}$ is sub-symmetric with respect to the indices $\{i_2,\ldots,i_m\}$. If $\bar {\bf x}\in {\rm SOL}(K,{\bf q},\mathcal{A})$ and
\begin{equation}\label{Secondcon}
\sum_{i_1,i_2,\ldots,i_m=1}^na_{i_1i_2i_3\ldots i_m}v_{i_1}v_{i_2}\bar x_{i_3}\cdots \bar x_{i_m}>0,~~~\forall~{\bf v}\in {\mathscr T}(\bar {\bf x},K)\cap {\mathbb S}(\bar {\bf x}),
\end{equation}
then $\bar {\bf x}$ is locally unique.
\end{proposition}

\begin{proof}
Assume for the sake of contradiction that $\bar {\bf x}$ is not locally unique.
There exists a sequence $\{{\bf x}^{(l)}\}\subset {\rm SOL}(K,{\bf q},\mathcal{A})$ converging to $\bar {\bf x}$ with ${\bf x}^{(l)}\neq \bar {\bf x}$ for
all $l$. By the definition of ${\mathscr T}(\bar {\bf x}, K)$, it is easy to see that every accumulation point of the normalized sequence
$\{({\bf x}^{(l)}-\bar {\bf x})/\|{\bf x}^{(l)}-\bar {\bf x}\|\}$ must belong to ${\mathscr T}(\bar {\bf x}, K)$. Let ${\bf v}$ be any such point. Without loss of generality, we assume $({\bf x}^{(l)}-\bar {\bf x})/\|{\bf x}^{(l)}-\bar {\bf x}\|\rightarrow {\bf v}$ as $l\rightarrow\infty$. Since $\bar {\bf x}$ and ${\bf x}^{(l)}\in {\rm SOL}(K,{\bf q},\mathcal{A})$, we have
\begin{equation}\label{barKK}
K\ni \bar {\bf x}\perp (\mathcal{A}\bar {\bf x}^{m-1}+{\bf q})\in K^*
\end{equation}
and
\begin{equation}\label{barKK1}
K\ni {\bf x}^{(l)}\perp \left(\mathcal{A}({\bf x}^{(l)})^{m-1}+{\bf q}\right)\in K^*,~~~\forall ~l=1,2,\ldots.
\end{equation}
By (\ref{barKK}) and (\ref{barKK1}), we have
\begin{equation}\label{lljk}
\bar {\bf x}^\top\left(\mathcal{A}({\bf x}^{(l)})^{m-1}+{\bf q}\right)\geq 0
\end{equation}
and
\begin{equation}\label{lljkl}
({\bf x}^{(l)})^\top\left(\mathcal{A}\bar {\bf x}^{m-1}+{\bf q}\right)\geq 0
\end{equation}
for every $l$. Consequently, by (\ref{barKK}) and (\ref{lljkl}), we  have
\begin{equation}\label{OPLK}
({\bf x}^{(l)}-\bar {\bf x})^\top\left(\mathcal{A}\bar {\bf x}^{m-1}+{\bf q}\right)\geq 0,
 \end{equation}
 and hence by diving by $\|{\bf x}^{(l)}-\bar {\bf x}\|$ in this expression and letting $l\rightarrow \infty$, we obtain
\begin{equation}\label{THRY}
{\bf v}^\top\left(\mathcal{A}\bar {\bf x}^{m-1}+{\bf q}\right)\geq 0.
\end{equation}
On the other hand, by (\ref{barKK1}) and (\ref{lljk}), it holds that
\begin{equation}\label{HJKL}
({\bf x}^{(l)}-\bar {\bf x})^\top\left(\mathcal{A}({\bf x}^{(l)})^{m-1}+{\bf q}\right)\leq 0
\end{equation} for $l=1,2,\ldots$.
Consequently, diving by $\|{\bf x}^{(l)}-\bar {\bf x}\|$ in (\ref{HJKL}) and letting $l\rightarrow \infty$, we obtain
\begin{equation}\label{GGHH}
{\bf v}^\top\left(\mathcal{A}\bar {\bf x}^{m-1}+{\bf q}\right)\leq 0.\end{equation}
since ${\bf x}^{(l)}\rightarrow \bar {\bf x}$ as $l\rightarrow\infty$. By (\ref{THRY}) and (\ref{GGHH}), we know that ${\bf v}^\top\left(\mathcal{A}\bar {\bf x}^{m-1}+{\bf q}\right)= 0$, that is, ${\bf v}\in {\mathbb S}(\bar {\bf x})$. Therefore, ${\bf v}\in {\mathscr T}(\bar {\bf x},K)\cap {\mathbb S}(\bar {\bf x})$. Since $\mathcal{A}$ is {\it sub-symmetric} with respect to the indices $\{i_2,\ldots,i_m\}$, by (\ref{HJKL}) and the fact that the function $\mathcal{A}{\bf x}^{m-1}$ is continuously differentiable at $\bar {\bf x}$, we know that for every $l$,
$$
\begin{array}{lll}
0&\geq &({\bf x}^{(l)}-\bar {\bf x})^\top\left(\mathcal{A}({\bf x}^{(l)})^{m-1}+{\bf q}\right)\\
&=&({\bf x}^{(l)}-\bar {\bf x})^\top\left(\mathcal{A}\bar {\bf x}^{m-1}+(m-1)\mathcal{A}\bar {\bf x}^{m-2}({\bf x}^{(l)}-\bar {\bf x})+o(\|{\bf x}^{(l)}-\bar {\bf x}\|)+{\bf q}\right)\\
&\geq &(m-1)({\bf x}^{(l)}-\bar {\bf x})^\top\left(\mathcal{A}\bar {\bf x}^{m-2}({\bf x}^{(l)}-\bar {\bf x})+o(\|{\bf x}^{(l)}-\bar {\bf x}\|)\right),
\end{array}
$$
where the last inequality is due to (\ref{OPLK}). Diving by $\|{\bf x}^{(l)}-\bar {\bf x}\|$ in the above expression  and letting $l\rightarrow \infty$, we obtain
$$
\sum_{i_1,i_2,\ldots,i_m=1}^na_{i_1i_2i_3\ldots i_m}v_{i_1}v_{i_2}\bar x_{i_3}\cdots \bar x_{i_m}={\bf v}^\top(\mathcal{A}\bar {\bf x}^{m-2}){\bf v}\leq 0,
$$
which is a contradiction. The proof is completed.
\qed\end{proof}

We are now at the stage of extending Theorem 7.3.12 of \cite{CPS92} to tensors, which is the main stability result of the problem under consideration.

\begin{theorem}\label{KHProp}
Let $K\in C(\mathbb{R}^n)$, $\bar {\bf q}\in \mathbb{R}^n$ and  $\bar{\mathcal{A}}\in {\mathcal T}_{m,n}$. Let $\bar {\bf x}\in {\rm SOL}(K,\bar {\bf q},\bar{\mathcal{A}})$. If $\mathcal{A}$ is sub-symmetric with respect to the indices $\{i_2,\ldots,i_m\}$ and \eqref{Secondcon} holds, then for every open neighborhood $\mathcal{N}$ of
$\bar {\bf x}$ satisfying \eqref{Soldsolution}, there exists a scalar $c>0$ such that, for all vectors ${\bf q}\in \mathbb{R}^n$ and tensors $\mathcal{A}\in {\mathcal T}_{m,n}$,
$$
\sup\left\{\|{\bf x}-\bar {\bf x}\|~|~{\bf x}\in {\rm SOL}(K,{\bf q},\mathcal{A})\cap{\rm cl}(\mathcal{N})\right\}\leq c\left(\|{\bf q}-\bar {\bf q}\|+\|\mathcal{A}-\bar{\mathcal{A}}\|_F\right).
$$
\end{theorem}

\begin{proof}
We argue by contradiction. Suppose that the conclusion is not true, then there exists a sequence $\{{\bf x}^{(l)}\}\subset {\rm cl}(\mathcal{N})$ and a sequence of $\{(\mathcal{A}^l,{\bf q}^{(l)})\}$ such that, for each $l$, ${\bf x}^{(l)}\in {\rm SOL}(K,{\bf q}^{(l)},\mathcal{A}^l)$ and  satisfies
\begin{equation}\label{xlAq}
\left\|{\bf x}^{(l)}-\bar {\bf x}\right\|>l\left(\|{\bf q}^{(l)}-\bar {\bf q}\|+\|\mathcal{A}^l-\bar{\mathcal{A}}\|_F\right).
\end{equation}
It clear that $\{{\bf x}^{(l)}\}$ is bounded and ${\bf x}^{(l)}$ is distinct from $\bar {\bf x}$ for every $l$. Consequently, by (\ref{xlAq}), we know that $\|{\bf q}^{(l)}-\bar {\bf q}\|+\|\mathcal{A}^l-\bar{\mathcal{A}}\|_F\rightarrow 0$ as $l\rightarrow\infty$ and
\begin{equation}\label{qaxxl}
\|{\bf q}^{(l)}-\bar {\bf q}\|=o(\|{\bf x}^{(l)}-\bar {\bf x}\|)~~~{\rm and}~~~\|\mathcal{A}^l-\bar{\mathcal{A}}\|_F=o(\|{\bf x}^{(l)}-\bar {\bf x}\|).
\end{equation}
Let ${\bf v}$ be any accumulation point of the normalized sequence
$\{({\bf x}^{(l)}-\bar {\bf x})/\|{\bf x}^{(l)}-\bar {\bf x}\|\}$. It is clear that ${\bf v}\in {\mathscr T}(\bar {\bf x}, K)$, since $\|{\bf x}^{(l)}-\bar {\bf x}\|\rightarrow 0$ by Proposition \ref{Convergsolution}. Since $\bar {\bf x}\in {\rm SOL}(K,\bar {\bf q},\bar{\mathcal{A}})$ and ${\bf x}^{(l)}\in {\rm SOL}(K,{\bf q}^{(l)},\mathcal{A}^l)$ for every $l$, we have
$$
K\ni \bar {\bf x}\perp (\bar{\mathcal{A}}\bar {\bf x}^{m-1}+\bar {\bf q})\in K^*~~{\rm and}~~K\ni {\bf x}^{(l)}\perp \mathcal{A}^l({\bf x}^{(l)})^{m-1}+{\bf q}^{(l)}\in K^*,~\forall ~l=1,2,\ldots.
$$
Moreover, since ${\bf x}^{(l)}\rightarrow \bar {\bf x}$ and $({\bf q}^{(l)},\mathcal{A}^l)\rightarrow (\bar {\bf q},\bar{\mathcal{A}})$ as $l\rightarrow\infty$, by a similar way used in Proposition \ref{localunique}, we obtain
${\bf v}^\top\left(\bar{\mathcal{A}}\bar {\bf x}^{m-1}+\bar {\bf q}\right)=0$, that is, ${\bf v}\in S(\bar {\bf x})$, and hence ${\bf v}\in {\mathscr T}(\bar {\bf x},K)\cap {\mathbb S}(\bar {\bf x})$. It is easy to see that
$$({\bf x}^{(l)}-\bar {\bf x})^\top(\mathcal{A}^l({\bf x}^{(l)})^{m-1}+{\bf q}^{(l)})\leq 0\quad \text{and} \quad({\bf x}^{(l)}-\bar {\bf x})^\top\left(\bar {\mathcal{A}}\bar {\bf x}^{m-1}+\bar {\bf q}\right)\geq 0.$$
 Consequently, since $\bar{\mathcal{A}}$ is {\it sub-symmetric} with respect to the indices $\{i_2,\ldots,i_m\}$, it holds that for every $l$,
$$
\begin{array}{lll}
0&\geq &({\bf x}^{(l)}-\bar {\bf x})^\top\left(\mathcal{A}^l({\bf x}^{(l)})^{m-1}+{\bf q}^{(l)}\right)\\
&=&({\bf x}^{(l)}-\bar {\bf x})^\top\left((\mathcal{A}^l-\bar {\mathcal{A}})({\bf x}^{(l)})^{m-1}+{\bf q}^{(l)}-\bar {\bf q}+\bar {\mathcal{A}}({\bf x}^{(l)})^{m-1}+\bar {\bf q}\right)\\
&=&({\bf x}^{(l)}-\bar {\bf x})^\top\left(\bar {\mathcal{A}}({\bf x}^{(l)})^{m-1}+\bar {\bf q}+o(\|{\bf x}^{(l)}-\bar {\bf x}\|)\right)\\
&=&({\bf x}^{(l)}-\bar {\bf x})^\top\left(\bar {\mathcal{A}}\bar {\bf x}^{m-1}+(m-1)\bar{\mathcal{A}}\bar {\bf x}^{m-2}({\bf x}^{(l)}-\bar {\bf x})+\bar {\bf q}+o(\|{\bf x}^{(l)}-\bar {\bf x}\|)\right)\\
&\geq &(m-1)({\bf x}^{(l)}-\bar {\bf x})^\top\left(\mathcal{A}\bar {\bf x}^{m-2}({\bf x}^{(l)}-\bar {\bf x})+o(\|{\bf x}^{(l)}-\bar {\bf x}\|)\right),
\end{array}
$$
where the second equality comes from (\ref{qaxxl}) and the boundedness of $\{{\bf x}^{(l)}\}$, and the third equality is due to the fact that $\bar{\mathcal{A}}{\bf x}^{m-1}$ is continuously differentiable at $\bar {\bf x}$. By diving by $\|{\bf x}^{(l)}-\bar {\bf x}\|^2$ in the above expression and letting $l\rightarrow \infty$, we know
$$
{\bf v}^\top\mathcal{A}\bar {\bf x}^{m-2}{\bf v}\leq 0.
$$
It is a contradiction. The proof is completed.
\qed\end{proof}

\section{Conclusions}\label{SecCon}
We consider the TCP over a general cone, which is an interesting generalization of LCP studied deeply in the literature. Unfortunately, the results of LCPs are not directly applicable to TCPs due to the nonlinearity of TCPs. However, as a special case of NCPs, in this paper, we first derive some specific forms of the solution existence results for TCP($K, {\bf q},\mathcal{A}$) from the results presented in \cite{GP92}. In particular, when the general cone $K$ reduces to a nonnegative cone ${\mathbb R}_+^n$ as discussed in the literature, we show the existence of a solution of TCP(${\mathbb R}_+^n, {\bf q},\mathcal{A}$) under copositivity, which is a weaker condition than the strict copositivity in previous papers. Moreover, we investigate the topological properties of the solution set SOL($K, {\bf q},\mathcal{A}$) and stability of TCP($K, {\bf q},\mathcal{A}$) at a given solution. These results are new and further enrich the theory of TCPs.

\medskip
\begin{acknowledgements}
The authors would like to thank Professor M. Seetharama Gowda and the two anonymous referees for their valuable comments on our manuscript and bringing our attention to the relevant reference \cite{DQW13,GP92,ZQZ14}.
This work was supported in part by National Natural Science Foundation of China at Grant No. 11571087 and Natural Science Foundation of Zhejiang Province at Grant Nos. (LZ14A010003, LY17A010028).
\end{acknowledgements}


\end{document}